\documentclass[12pt,a4paper,titlepage]{scrartcl}
\usepackage[utf8]{inputenc}
\usepackage[english, russian]{babel}

\usepackage{amsmath}
\usepackage{amssymb, wasysym, amsfonts, latexsym, amsthm}
\usepackage{pgf,tikz}
\usepackage{mathrsfs}
\usetikzlibrary{arrows}

\newtheorem{defi}{Определение}
\newtheorem{teor}{Теорема}
\newtheorem{lem}{Лемма}[section]
\newtheorem{pre}{Предложение}[section]
\newtheorem{gip}{Гипотеза}

\begin{document}

\textbf{\large Зверев Иван Сергеевич}

\smallskip

\textbf{\large Разрезания трапеций на трапеции, гомотетичные трапециям данного набора}

\smallskip

{\large Национальный исследовательский университет <<Высшая школа экономики>>.}

\bigskip

В данной работе доказаны три теоремы, связанные с разрезанием трапеций на трапеции, гомотетичные данным.

В теореме \ref{my_big} доказано, что гомотетиями трапеции с рациональным отношением оснований можно замостить любую трапецию с рациональным отношением оснований и такими же углами, но нельзя замостить никаких других трапеций. 

Далее я перехожу к трапециям, отношение оснований которых является квадратичной иррациональностью. В теореме \ref{twodim} для некоторых пар трапеций доказывается, что их гомотетиями можно замостить любую трапецию с такими же углами и отношением оснований из того же квадратичного поля. В теореме \ref{log} ещё для некоторых трапеций с квадратично-иррациональным отношением оснований приведено необходимое условие на трапеции, которые можно ими замостить. Это условие примечательно тем, что содержит трансцендентную функцию. Это первое появление трансцендентной функции в задачах о замощении многоугольников подобными многоугольниками.

\section{Введение}

\footnotetext{Автор поддержан грантом Президента РФ МК-6137.2016.1}

Первая сложная проблема, связанная с замощениями подобными многоугольниками, была рассмотрена Максом Деном в 1903 году и заключалась в замощении прямоугольников квадратами \cite{german}. Ден доказал, что квадратами можно замостить только прямоугольник с рациональным отношением сторон.

В 1940-ом году вышла статья Брукса, Смита, Стоуна и Татта о разрезании прямоугольника на квадраты попарно различного размера \cite{BST}. В ней было доказано, что прямоугольник с рациональным отношением сторон всегда можно разрезать на попарно различные квадраты, причём бесконечным числом способов. Также именно в ней впервые имело место сопоставление замощений прямоугольника квадратами и электрических цепей, при помощи которого с большей легкостью передоказывалась теорема Дена (определение электрических цепей, как и их применение, вы найдёте в разделе \ref{ken_proof} данной работы).

В дальнейшем задача о разрезании прямоугольника на прямоугольники была в достаточной степени обобщена и исследована. В 1994-ом году Фрайлинг, Ласкович, Секереш и Ринн сформулировали критерий, при выполнении которого квадрат можно разрезать на прямоугольники с данным отношением сторон \cite{FL1,FL2}. В 1997-ом они обобщили свой результат и привели критерий того, что данный прямоугольник замощается прямоугольниками с данным отношением сторон \cite{FLR2}.

Наконец, стоит отметить недавнюю статью Шарова \cite{Sharov}, в которой для прямоугольников с отношением сторон из $\mathbb Q[\sqrt d]$ доказан критерий того, что данный прямоугольник можно замостить прямоугольниками с отношением сторон \emph{из данного множества} (в статьях Фрайлинга-Ласковича-Секереша-Ринна отношение сторон замощающих прямоугольников было фиксировано). Я в своей работе также обращаюсь к полю $\mathbb Q[\sqrt d]$, поэтому ниже приведу полную формулировку этого критерия для сравнения с результатами, полученными для трапеций.

Также были известны теоремы о замощениях другими многоугольниками; в частности, в статье Сегеди \cite{triangles} выведен критерий того, что квадрат замощается прямоугольными треугольниками с фиксированным углом.

Теперь я хочу перейти к статье Кеньона \cite{ken}, больше остальных повлиявшей на мою работу. Кеньон несколько изменяет модель электрических цепей, использованную ранее, так что теперь ей можно описать не только замощения прямоугольниками, но и замощения трапециями.

Важный результат Кеньона состоит в том, что если отношения оснований к высоте для замощающих трапеций рациональны, то и отношения оснований к высоте замощённой ими трапеции рациональны. Теорема $\ref{kenyon_big}$ ниже -- это чуть более общий вариант теоремы Кеньона, для отношений оснований к высоте, лежащих в некотором подполе $\mathbb K\subset \mathbb R$. В таком виде она будет необходима мне в дальнейшем. Чтобы продемонстрировать связь электрических цепей с разрезаниями трапеций, я привожу в разделе $\ref{ken_proof}$ по сути принадлежащее Кеньону доказательство этой теоремы.

\bigskip

\begin{teor}\cite[следствие 7]{ken}
\label{kenyon_big}
Пусть $A_1, A_2, ..., A_n$ -- трапеции, все основания которых параллельны. Пусть для любой из этих трапеций отношение каждого основания к высоте лежит в подполе $\mathbb K\subset \mathbb R$. Тогда, если трапецию $B$ можно замостить трапециями, гомотетичными $A_1, A_2, ..., A_n$, то отношения оснований трапеции $B$ к высоте тоже лежат в $\mathbb K$.
\end{teor}

\bigskip

Самая общая формулировка вопроса, на который я стремлюсь ответить: <<Какие трапеции можно замостить трапециями, гомотетичными трапециям данного набора>>? В настолько общей формулировке вопрос вряд ли допускает обозримый ответ, поэтому для набора приходится вносить ограничения.

Во-первых, я хочу рассматривать наборы, все основания трапеций которого параллельны. Кеньону это условие было необходимо, чтобы работала физическая модель (расстояния между прямыми, содержащими основания, сопоставлялись разности потенциалов в цепи, см. лемму $\ref{odd}$ ниже).

Во-вторых, я хочу рассматривать наборы, у всех трапеций которых одинаковы углы при основании. Теорема Кеньона верна и для трапеций с разными углами, однако я собираюсь параметризовать трапеции числами, поэтому класс рассматриваемых трапеций приходится ограничивать.

В-третьих, я буду рассматривать наборы трапеций, у которых оба угла при основании равны $45^o$. Это служит упрощению обозначений -- решив задачу для равнобедренных трапеций с углами по $45^o$, результаты нетрудно переформулировать для других трапеций, воспользовавшись афинными преобразованиями.

\bigskip

Для трапеций с рациональным отношением оснований всё просто: любая такая трапеция может замостить любую другую с такими же углами при основании (в теореме \ref{my_big} написано <<с рациональными средними линиями>>, однако для трапеций единичной высоты с углами $45^o$ эти свойства равносильны). Тем не менее, построить конструкцию замощения здесь сложнее, чем при доказательстве того же факта для прямоугольников (см. рис. \ref{e1}). Следующая теорема является новой.

\begin{teor}[о рациональных средних линиях] \label{my_big}
Пусть $A, B$ -- трапеции единичной высоты с параллельными основаниями, углами $45^o$ при основании и средними линиями $a,b,$ соответственно. Пусть $a\in \mathbb Q$. Тогда трапециями, гомотетичными $A$, можно замостить трапецию $B$, если и только если $b\in\mathbb Q$.
\end{teor}

\begin{figure}[h!]
\centering
\includegraphics[scale=0.8]{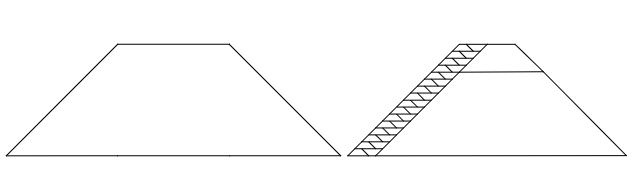}
\caption{Пример к теореме \ref{my_big}: гомотетии трапеции со средней линией 2 замощают трапецию со средней линией $\frac32$ (обе трапеции -- высоты $1$ и с углами $45^o$ при основании)}\label{e1}
\end{figure}

Следующие теоремы связаны с квадратичным расширением $\mathbb Q$, поэтому я введу необходимые определения уже здесь.

\begin{defi}Пусть $d > 0, d\in \mathbb Q, \sqrt d \not\in\mathbb Q$. Тогда определим множество $\mathbb Q[\sqrt d]$ как все числа вида $a+b\sqrt d$, где $a,b\in \mathbb Q$.
\end{defi}
Ясно, что всякий $x\in\mathbb Q[\sqrt d]$ представляется в виде $a+b\sqrt d$ для $a,b\in \mathbb Q$ единственным образом. Это делает следующее определение корректным.
\begin{defi}
Пусть $x\in\mathbb Q[\sqrt d]$ и $x = a+b\sqrt d$ для $a,b\in\mathbb Q$. Тогда обозначим через $\overline x$ число $a-b\sqrt d$.
\end{defi}
В случае с трапециями, отношения оснований которых лежат в $\mathbb Q[\sqrt d]$, очевидно, что не для любой трапеции её гомотетиями можно замостить все остальные. Это следует, как минимум, из теоремы \ref{my_big}. Однако, следующая теорема показывает, насколько легко найти две трапеции, гомотетии которых могут замостить любую трапецию с отношением оснований из $\mathbb Q[\sqrt d]$.

\begin{teor}[о средних линиях из $\mathbb Q[\sqrt d\rbrack$]
\label{twodim}

Пусть $d > 0, d\in\mathbb Q,\sqrt d\not\in\mathbb Q$. Пусть $A,B,C$ -- трапеции единичной высоты c параллельными основаниями, углами $45^o$ при основании, и длины их средних линий равны соответственно $a,b,c$. Пусть также $a,b\in\mathbb Q[\sqrt d]$ и $\overline a > 0$, $\overline b < 0$. Тогда трапециями, гомотетичными $A$ и $B$, можно замостить $C$, если и только если $c\in\mathbb Q[\sqrt d]$.
\end{teor}

Здесь стоит процитировать теорему Шарова, чтобы провести необходимые параллели с замощениями прямоугольников:

\begin{teor}\cite[теорема 1]{Sharov}.
\label{she}
Пусть $x_1=a_1+b_1\sqrt p, ..., x_n = a_n+b_n\sqrt p$ -- такие числа, что $x_i>0, a_i,b_i,p\in\mathbb Q\ (1\le i\le n)$ и $\sqrt p \not\in \mathbb Q$. Тогда:

1) если существуют такие числа $i$ и $j$, что $1\le i,j\le n$ и $(a_i-b_i\sqrt p)(a_j-b_j\sqrt p) < 0$, то прямоугольник с отношением сторон $z$ можно разрезать на прямоугольники с отношениями
сторон $x_1,...,x_n$ тогда и только тогда, когда
$$z\in \{e+f\sqrt p > 0 | e,f\in \mathbb Q\}$$

2) если для всех $i$ ($1\le i\le n$) $a_i-b_i\sqrt p>0$, то прямоугольник с
отношением сторон $z$ можно разрезать на прямоугольники с отношениями сторон $x_1,...,x_n$ тогда и только тогда, когда

$$z\in \left\{e+f\sqrt p | e,f\in \mathbb Q,  e>0, \frac{|f|}{e} \le \max_i\frac{|b_i|}{a_i}\right\}$$

3) если для всех $i$ ($1\le i\le n$) $a_i-b_i\sqrt p<0$, то прямоугольник с
отношением сторон $z$ можно разрезать на прямоугольники с отношениями сторон $x_1,...,x_n$ тогда и только тогда, когда

$$z\in \left\{e+f\sqrt p | e,f\in \mathbb Q,  f>0, \frac{|e|}{f} \le \max_i\frac{|a_i|}{b_i}\right\}$$

\end{teor}

Теорема \ref{twodim} -- это аналог случая 1) для трапеций в теореме \ref{she}. Если мы будем рассматривать замощения гомотетиями трапеций данного набора, то найдя в наборе две трапеции со средними линиями $a,b$ такими что $\overline a\overline b < 0$, мы замостим их гомотетиями любую трапецию со средней линией $c$, такой что $c\in\mathbb Q[\sqrt d], c > 1$. То, что всеми трапециями набора мы не замостим больше, чем данными двумя, следует, например, из теоремы \ref{kenyon_big}.

Наша последняя теорема \ref{log} -- это необходимое условие на тривиальное разрезание некоторых трапеций с отношением оснований из $\mathbb Q[\sqrt d]$ на трапеции, гомотетичные данной. Тривиальные разрезания, о которых в ней говорится -- это разрезания особого вида, которые я определю в разделе \ref{tri}. У меня было две причины перейти к ним -- во-первых, по опыту исследований прямоугольников можно выдвинуть гипотезу, что если сушествует разрезание трапеции на гомотетичные заданным, то существует и тривиальное разрезание. Во-вторых, для них есть вычислительно сложный, но гарантированный способ доказать незамостимость одной трапеции другой. В будущих работах я собираюсь сначала найти достаточное условие тривиальной замостимости одной трапеции гомотетиями другой, а затем уже перейти от тривиальных разрезаний к произвольным.

Замечу также, что этот результат показывает, что у случаев 2) и 3) теоремы Шарова при переходе от прямоугольников к трапециям не появляется очевидных аналогов.

Я считаю этот результат одним из самых важных, потому что в нём впервые появляется трансцендентная функция (логарифм) от отношения оснований трапеции. Ранее в результатах о замощениях многоугольников подобными многоугольниками возникали только рациональные функции от сторон (в качестве примеров приведу \cite[теорема 1]{Sharov}, \cite[теорема 10]{FL2},\cite[теорема 5]{FLR2}).

\begin{teor}\label{log}
Пусть $A, B$ -- трапеции единичной высоты c параллельными основаниями, углами $45^o$ при большем основании и средними линиями $a, b\in \mathbb Q[\sqrt d]$, причём $1 < \overline a < a$. Пусть трапециями, гомотетичными трапеции $A$, можно тривиально замостить трапецию $B$.
Тогда $b$ удовлетворяет следующим условиям:

(i) $1 < \overline b < b$;

(ii) $\overline b/ b \ge \overline a/a$;

(iii) $\ln G(b)/ \ln \overline {G(b)} \ge \ln G(a)/ \ln \overline{G(a)} $, где $G(x) = (x-1)/(x+1).$
\end{teor}

Следующее предложение показывает, что несторогое неравенство (iii) в теореме \ref{log} является точным, и к тому же обращается в равенство в бесконечном количестве случаев.

\begin{pre}\label{last}
Пусть $A$ -- трапеция единичной высоты с углами $45^o$ при большем основании и средней линией $a\in \mathbb Q[\sqrt d]$, причём $1 < \overline a < a$. 

Тогда сущетвует бесконечно много различных $b\in \mathbb Q[\sqrt d]$, таких что выполнены условия: \begin{itemize}\item $\ln G(b)/ \ln \overline {G(b)} = \ln G(a)/ \ln \overline{G(a)} $ \item Трапециями, гомотетичными трапеции $A$, можно тривиально замостить трапецию единичной высоты с углами $45^o$ при большем основании и средней линией $b$.
\end{itemize}
\end{pre}

Основываясь на предложении \ref{last}, я выдвигаю гипотезу о достаточности указанных неравенств.

\begin{gip}\label{gi}
Пусть $A, B$ -- трапеции единичной высоты c параллельными основаниями, углами $45^o$ при большем основании и средними линиями $a, b\in \mathbb Q[\sqrt d]$, причём $1 < \overline a < a$. Пусть $b$ удовлетворяет условиям из теоремы \ref{log}. Тогда трапециями, гомотетичными трапеции $A$, можно замостить трапецию $B$.
\end{gip}

Для лучшего понимания доказательства теорем \ref{twodim} и \ref{log} и их геометрического смысла я рекомендую обратиться к разделу \ref{geom} данной работы.

\section{Организация работы}
В работе по порядку доказывается четыре большие теоремы, а именно теоремы \ref{kenyon_big}, \ref{my_big}, \ref{twodim} и \ref{log}. Разделы, выделенные непосредственно на доказательства, перемежаются разделами, содержащими вспомогательные конструкции.

В разделе \ref{ken_proof} я привожу полное доказательство теоремы \ref{kenyon_big}, без больших изменений переписанное из работы Кеньона. В том числе, в этот раздел входит описание понятия <<электрической цепи>>, необходимое в данном доказательстве (но ненужное для понимания доказательства остальных теорем).

В разделе \ref{notion} вводятся удобные нам обозначения для геометрических фигур, а также несложные леммы о замощениях, сформулированные в терминах этих обозначений. Эти леммы и определения многократно используются в доказательствах теорем \ref{my_big}, \ref{twodim} и \ref{log}, особенно часто приходится ссылаться на лемму \ref{base}.

В разделе \ref{me_proof} доказана теорема \ref{my_big}. С использованием лемм из раздела \ref{notion} и теоремы \ref{kenyon_big} доказательство получается коротким и несложным.

\bigskip

Далее идут разделы \ref{hlems}, \ref{di_proof}, по которым разнесено доказательство теоремы \ref{twodim}. Конструкция доказательства уже гораздо сложнее, чем для теоремы \ref{my_big}, хотя методы и похожи. Отдельную сложность может представлять неясный интуитивно смысл функции $h$ или описывающих её лемм (\ref{h1}, \ref{h2}, \ref{h3}). Для лучшего понимания стоит обращаться к разделу \ref{geom}.

В разделе \ref{tri} определяются тривиальные разрезания и вводятся леммы связанные с ними. 

В разделе \ref{lo_proof} доказана теорема \ref{log}, в которой используется понятие тривиальных разрезаний.

Раздел \ref{geom} не содержит новых результатов, но поясняет старые. При работе с полем $\mathbb Q[\sqrt d]$ можно представлять число $a\in\mathbb Q[\sqrt d]$ как точку $(\overline a, a)$ на координатной плоскости. Это представление придаёт утверждениям из разделов \ref{hlems}, \ref{di_proof}, \ref{lo_proof} наглядный геометрический смысл.

\section {Доказательство теоремы \ref{kenyon_big}}\label{ken_proof}

Этот раздел не содержит новых результатов и, по сути, повторяет доказательство из работы \cite[следствие\ 7]{ken} с тем замечанием, что вместо $\mathbb Q$ мы рассматриваем произвольное подполе $\mathbb K\subset \mathbb R$.

\smallskip

\begin{defi}\label{zep}\emph{\emph{Электрическая цепь} -- это ориентированный взвешенный граф с положительными весами рёбер и двумя выделенными вершинами $N$ и $P$; при этом для любой вершины $x$, кроме $N$ и $P$, существуют ориентированные пути от $x$ до $N$ и от $x$ до $P$. Количество рёбер, идущих из $x$ в $y$, будет обозначаться через $c_{xy}$, вес $k$-того из них через $[xy]_k$.}
\end{defi}

\begin{defi}\label{pot}\emph{\emph{Потенциал} -- это вещественнозначная функция $w$ на множестве вершин электрической цепи, удовлетворяющая следующим условиям:}

1) $w(N) = 1, w(P) = 0$;

2) Для любого $x$, не равного $N$ и $P$, верно

$$\sum_{y\not = x}\sum_{i=1}^{c_{xy}}[xy]_i(w(x)-w(y))=0.$$
\end{defi}

\begin{defi} (см. рис. \ref{traa})
 \emph{Для произвольного замощения трапеции трапециями, все основания которых параллельны, обозначим через $S$ множество всех точек всех оснований трапеций замощения. Множество $S$ представимо в виде объединения нескольких непересекающихся отрезков (очевидно, единственным образом). Отрезки такого разбиения будем называть \emph{разрезами} (на рис. \ref{traa} изображены жирными линиями). В частности, оба основания замощённой трапеции всегда являются разрезами.}
\end{defi}

В пример приведу рисунок \ref{traa}. На нём трапеция $ACMJ$ разбита на шесть различных трапеций. Объединение их оснований можно представить в виде объединения непересекающихся отрезков только одним способом: $AC\sqcup DG\sqcup HI \sqcup JM$ (выделены жирным). Поэтому для данного разбиения имеем четыре разреза.

\begin{figure}[h!]
\centering
\includegraphics[scale=0.6]{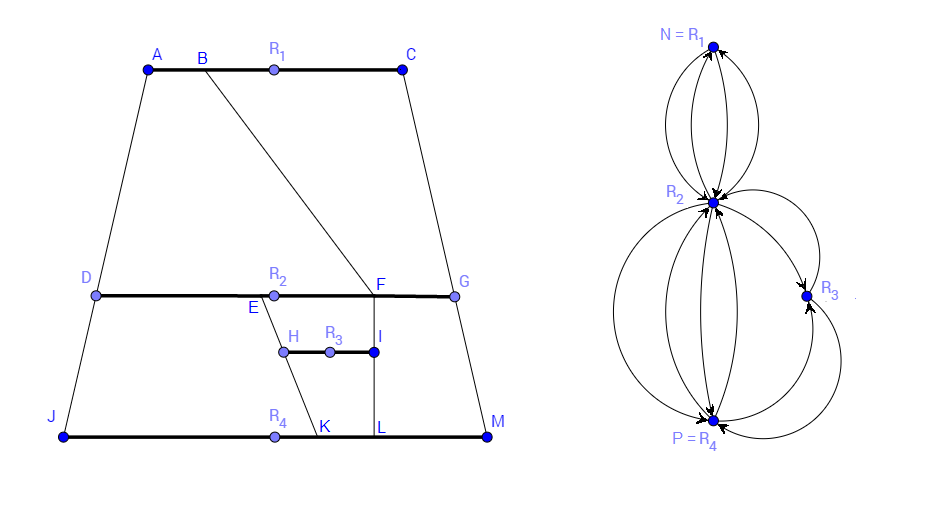}
\caption{Замощение трапеции и его электрическая цепь, к построению \ref{popo}}\label{traa}
\end{figure}

В дальнейшем рассуждении (как и во всей работе) мы говорим о трапециях, все основания которых параллельны. Для удобства, будем считать, что их основания расположены горизонтально. Как следствие, мы можем говорить о <<нижнем>> и <<верхнем>> основании любой трапеции замощения.

\bigskip

\textbf{Общий план доказательства теоремы \ref{kenyon_big}}.

\begin{enumerate}
\item По замощению трапеции высоты $1$ другими трапециями, мы строим электрическую цепь, веса рёбер которой равны отношениям оснований к высоте для трапеций замощения (см. определение \ref{popo} и лемму \ref{odd}).

\item Показываем, что в цепи существует потенциал $w$, такой что высота каждой трапеции замощения является разностью значений $w$ на некоторых вершинах цепи (предложение \ref{good_w} и лемму \ref{odd}).

\item Доказываем, что потенциал в произвольной электрической цепи единственен и выражается через веса рёбер рационально с рациональными коэффициентами (лемма \ref{ratpot}).

\end{enumerate}

Из 1-3 почти напрямую будет следовать утверждение теоремы.
\begin{defi}\label{popo} Пусть дано замощение трапеции трапециями с горизонтальными основаниями. Тогда \emph{цепь, соответствующая данному замощению} строится следующим образом.

1) Множество вершин графа -- это множество середин всех разрезов замощения (см. $R_1,R_2,R_3,R_4$ на рис. \ref{traa}).

2) Выделенными вершинами будут $P$ -- середина нижнего основания и $N$ -- середина верхнего основания замощаемой трапеции.

3) Множество рёбер графа строим, последовательно добавляя рёбра (вначале рёбер нет). Для этого в произвольном фиксированном порядке рассмотрим все трапеции замощения.

Пусть мы рассматриваем трапецию $T$. Обозначим длину её нижнего основания через $b_1$, верхнего -- через $b_2$. Обозначим середины разрезов, содержащих эти основания, через $B_1$ и $B_2$ соответственно. Обозначим высоту трапеции $T$ через $h$. Тогда добавим в граф ребро из $B_1$ в $B_2$ весом $b_1/h$ и ребро из $B_2$ в $B_1$ весом $b_2/h$.

Таким образом по замощению мы получили ориентрованный граф с взвешенными рёбрами.
\end{defi}

На рисунке \ref{traa} можно увидеть построенную указанным способом цепь для четырёх разрезов (как следствие, четырёх вершин) и шести трапеций замощения (как следствие, 12-ти ориентированных рёбер). Чтобы не рисовать получившийся граф поверх замощения, точки $R_1,R_2,R_3,R_4$ изображены дважды.

\begin{pre}\label{is}
Для любого замощения трапеции трапециями, цепь, соответствующая данному замощению (определение \ref{popo}) действительно является электрической цепью.
\end{pre}

\begin{proof}
Из определения \ref{popo} очевидно, что веса рёбер положительны. Тогда, по определению \ref{zep}, остаётся доказать, что из произвольной вершины $K$ есть пути в $N$ и в $P$. Здесь доказывается, что есть путь в $P$, так как для $N$ рассуждение аналогично.

Доказательство от противного: пусть из $K$ нет пути в $P$. Тогда обозначим через $Q \not = P$ самую низкую (одну из самых низких) вершину графа, такую что из $K$ есть путь в $Q$.

Пусть $Q_1Q_2 \supset Q$ -- разрез, содержащий $Q$. Так как он не является нижним основанием замощаемой трапеции, то существует трапеция замощения, верхнее основание которой принадлежит $Q_1Q_2$. Нижнее основание такой трапеции, соответственно, принадлежит некоторому разрезу $L_1L_2$ с серединой $L$. Но в этом случае:

1) Из $Q$ в $L$ есть ребро, порождённое рассмотренной трапецией. Как следствие, из $K$ есть путь в $L$.

2) Точка $L$ строго ниже, чем точка $Q$.

Это противоречит тому, что $Q$ -- самая низкая вершина, до которой есть путь из вершины $K$.

Достигнуто противоречие, значит путь из $K$ в $P$ существует.
\end{proof}

\begin{pre}
\label{good_w}
Возьмём произвольное замощение трапеции трапециями с горизонтальными основаниями. Определим на вершинах цепи, соответствующей этому замощению, функцию $w$ как расстояние от данной вершины до прямой, содержащей нижнее основание замощаемой трапеции.

Тогда $w$ является потенциалом в этой цепи.
\end{pre}

\begin{proof}
Определение потенциала \ref{pot} состоит из двух пунктов. Проверим их по очереди.
\begin{enumerate}
\item
Вершины $P$ и $N$ -- это середины нижнего и верхнего оснований трапеций. Для них очевидно, что $w(N) = 1, w(P)=0$.
\item
Пусть $x,y$ -- вершины цепи и середины разрезов $X_1X_2, Y_1Y_2$ соответственно. Рассмотрим выражение $[xy](w(x)-w(y))$, как в определении \ref{pot} (я опускаю индекс $i$, так как для кратных рёбер рассуждение аналогично).

Если $X$ и $Y$ соединены ребром, то, по определению \ref{popo}, это ребро соответствует некоторой трапеции $T$ с основаниями на $X_1X_2, Y_1Y_2$. Если обозначить длины оснований трапеции $T$ через $b_x, b_y$, а высоту через $h$, то, опять же, по определению \ref{popo}, вес $[xy]$ будет равняться $b_x/h$.

С другой стороны, $w(x)-w(y)$ -- это разность расстояний от $x$ и $y$ до прямой, содержащей нижнее основание замощённой трапеции. Как следствие, это высота трапеции $T$ с точностью до знака, то есть $\pm h$.

Таким образом, $[xy](w(x)-w(y)) = \pm (b_x/h)h = \pm b_x$. Нетрудно заметить, что знак, с которым будет взят $b_x$, будет зависеть от того, лежала ли вершина $y$ выше или ниже, чем вершина $x$.

Если мы просуммируем выражения $[xy](w(x)-w(y))$ по всем $y$, таким что из $x$ в $y$ есть ребро и $y$ лежит ниже, чем $x$, то получится сумма верхних оснований лежащих на $X_1X_2$ для всех трапеций, одно основание которых лежит на $X_1X_2$, а второе ниже чем $X_1X_2$. Эта сумма как раз равна длине $X$.

Если просуммировать те же выражения, но для $y$ более высоких, чем $x$, получится длина $X_1X_2$, взятая с минусом. Поэтому общая сумма будет равна нулю, как и требуется в определении.

\end{enumerate}

В качестве примера, опять возьмём рисунок \ref{traa}. Из вершины $R_2$ исходит пять рёбер -- два в вершину $R_1$, два в вершину $R_4$ и одно в вершину $R_3$. При этом только вершины $R_4$ и $R_3$ находятся ниже, чем $R_2$.

Ребру, идущему в $R_3$, соответствует трапеция $EFIH$. Выражение $[R_2R_4][w(R_2)-w(R_4)]$ будет равняться длине основания $EF$. Аналогичные выражения от рёбер, идущих в $R_3$, равняются длинам оснований $DE$ и $FG$. Поэтому в сумме три выражения дают длину разреза $DG$.

\end{proof}

Прежде чем доказывать общие факты об электрической цепи, сгруппируем предложения выше в одну лемму, которую потом будет удобнее использовать.

\begin{lem}\label{odd}
Для произвольного замощения трапеции трапециями с горизонтальными основаниями существует электрическая цепь c определённым на ней потенциалом, удовлетворяющая двум условиям:

1) Вес каждого ребра цепи равен отношению основания некоторой трапеции замощения к её высоте.

2) Высота каждой трапеции замощения равна разности потенциалов в двух точках цепи.
\end{lem}

\begin{proof}
Необходимая цепь строится по определению \ref{popo}, из него же следует условие 1). Условие 2) следует напрямую из предложения \ref{good_w}.
\end{proof}

Остаётся доказать одну лемму о потенциале в произвольной электрической цепи.

\begin{lem}
\label{ratpot}
Если на электрической цепи существует потенциал, то он единственен, и его значение в любой вершине можно представить как рациональную функцию с рациональными коэффициентами от весов рёбер цепи.
\end{lem}

\begin{proof} Докажем первое утверждение леммы \ref{ratpot}. Пусть у цепи есть два потенциала. Обозначим их через $w$ и $v$. Рассмотрим функцию $f = w-v$.

Из определения потенциала следует, что для неё верно $$\sum_{y\not = x}\sum_{i=1}^{c_{xy}}[xy]_i(f(x)-f(y))=0$$ при $x$ не равных $N$ и $P$.

Докажем, что $f$ достигает максимума в вершине $N$ или $P$. Пусть $f$ достигла максимума в точке $x\not\in \{N, P\}$. Рассмотрим ориентированный путь $x_0,..,x_n$ от $x=x_0$ до $N=x_n$. Докажем по индукции, что для $k=0,..,n$ выполнено $f(x_k)=f(x)$ (чтобы получить $f(N)=f(x)$ и достижение максимума в $N$). База очевидна, остаётся переход.

При $k < n$ верно равенство $$\sum_{y\not = x_k}\sum_{i=1}^{c_{x_ky}}[x_ky]_i(f(x_k)-f(y))=0.$$ Так как $f(x_k)=f(x)$ максимален, то все слагаемые слева неотрицательны, как следствие, равны нулю. Так как между $x_k$ и $x_{k+1}$ есть ребро, то необходимо чтобы $f(x_k)-f(x_{k+1})=0$, то есть $f(x_{k+1})=f(x)$. Переход доказан, отсюда $f(N)=f(x)$ и максимум достигнут в $N$. Аналогичное рассуждение доказывает, что минимум $f$ достигается на $N$ или на $P$.

Однако, из условия 1) в определении \ref{pot} следует, что $f(N)=f(P)=0$. А если минимум и максимум равны $0$, то функция $f$ равна нулю на всей цепи. Значит, взятые изначально $w$ и $v$ совпадают. Таким образом, потенциал единственен.

Теперь докажем второе утверждение леммы. 

Функцию на множестве мощности $n$ можно представить набором из $n$ чисел: $w(x_1), ... , w(x_n)$. Чтобы этот набор соотвествовал некоторому потенциалу цепи, на него надо наложить несколько условий следующего вида:
\begin{itemize}

\item $w(x_i) = 1$ для некоторых $i$,

\item$w(x_i) = 0$ для некоторых $i$

\item$\Sigma_{j\not=i}\Sigma_k^{c_{x_ix_j}}[x_ix_j]_k(w(x_i)-w(x_j))=0$ для некоторых $i$.
\end{itemize}
Заметим, что все условия такого вида -- линейные уравнения относительно $w(x_i)$, то есть потенциалом является любое решение некоторой системы линейных уравнений. Коэффициенты этих уравнений -- веса рёбер графа, а также $0$ и $1$.

Как мы знаем из первой части леммы, потенциал в цепи всего один, а значит, и решение у системы одно. Единственное решение системы можно получить методом Гаусса. При вычислении методом Гаусса задействуются только четыре основным арифметических действия, поэтому решение системы будет значением рациональной функции с рациональными коэффициентами от весов рёбер графа. \end{proof}

\begin{proof}[Доказательство теоремы \ref{kenyon_big}] Подействуем на указанное в теореме замощение гомотетией так, чтобы высота замощённой трапеции стала равна 1. Отношения оснований трапеций замощения к высотам всё ещё лежат в поле $\mathbb K$.

Возьмём цепь, описанную в лемме \ref{odd}. По той же лемме, веса ребёр в этой цепи будут лежать в $\mathbb K$. Тогда, по лемме \ref{ratpot}, значения потенциала тоже будут лежать в $\mathbb K$. Из свойств цепи, высоты всех трапеций замощения лежат в $\mathbb K$. Зная, что в $\mathbb K$ также лежат отношения оснований к высотам, получаем, что все основания трапеций замощения тоже лежат в $\mathbb K$.

Оба основания замощаемой трапеции представляются в виде суммы оснований замощающих трапеций, поэтому тоже лежат в $\mathbb K$. Высота равна единице, поэтому отношения оснований к высоте замощённой трапеции тоже лежат в $\mathbb K$, что и требовалось доказать.
\end{proof}

\section{Основные геометрические определения и связанные с ними леммы}\label{notion}

\begin{defi}\label{stan}
\emph{Стандартная трапеция} -- это трапеция единичной высоты с горизонтальными основаниями и углами $45^o$ при нижнем основании. Стандартную трапецию со средней линией $a$ для краткости обозначим через $t(a)$.\end{defi}
\begin{defi}\emph{Стандартный параллелограмм} -- это параллелограмм единичной высоты с двумя горизонтальными основаниями и левым нижним углом $45^o$ при основании. Стандартный параллелограмм с основанием $a$ обозначим через $p(a)$.
\end{defi}

В некоторых случаях я буду писать <<стандартная фигура>> -- это следует понимать как стандартный параллелограмм или стандартную трапецию.

Как нетрудно заметить, средняя линия стандартной трапеции определяет её однозначно -- по ней можно высчитать оба основания. Аналогично, основание стандартного параллелограмма однозначно задаёт параллелограмм.

\bigskip
В работе мне неоднократно понадобится функция, связывающая среднюю линию стандартной трапеции с отношением оснований. Поэтому я сформулирую её свойства заранее.

\textbf{Обозначение.} Обозначим через $G(x)$ функцию, определённую на интервале $(1,+\infty)$ как $G(x):=(x-1)/(x+1)$.

\begin{lem}[свойства функции G]\label{glem}

1) Функция $G$ -- биекция между интервалами $(1,+\infty)$ и $(0,1)$. В частности, она обратима и $G^{-1}(x)=(1+x)/(1-x)$.

2) $G(x)$ возрастает.

3) $G(x)$ -- рациональная функция с рациональными коэффициентами; как следствие, если $x$ принадлежит полю $\mathbb K$, то $G(x)$ принадлежит тому же полю $\mathbb K$.

4) Если стандартная трапеция имеет среднюю линию $a$, то её отношение оснований (меньшего к большему) равно $G(a)$. Также верно обратное.
\end{lem}

\begin{proof}
Все алгебраические факты очевидны. С геометрией тоже просто: стандартные трапеции имеют единичную высоту и углы $45^o$ при основании, поэтому при средней линии в $a$, её основания будут равны как раз $(a-1)$ и $(a+1)$.
\end{proof}

Теперь введём леммы, которые помогут нам избавиться от геометрическиех построений и далее оперировать исключительно числовыми понятиями.

\begin{lem} \label{draw}
Пусть $A,B,C$ -- стандартные фигуры, и $C$ состоит из двух фигур, гомотетичных $A$ и $B$ соответственно. Тогда существуют числа $a,b$, для которых выполнено одно из следующих условий 1)-5):

1) $A = p(a), B = p(b), C = p(a+b)$.

2) $A = p(a), B = p(b), C = p(ab/(a+b))$

3) $A = t(a), B = t(b), C = p(a+b)$

4) $A = p(a), B = t(b), C = t(a+b)$

5) $A = t(a), B = t(b), C = t(G^{-1}(G(a)G(b))) = t((ab+1)/(a+b))$
\end{lem}

\begin{figure}[h!]
\centering
\includegraphics[scale=0.8]{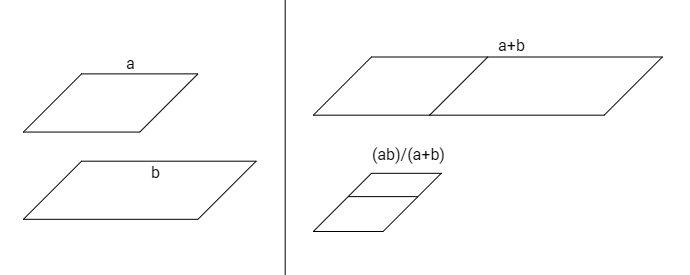}
\caption{К доказательству леммы \ref{draw} -- выполнены условия 1) и 2)}\label{k1}
\end{figure}

\begin{figure}[h!]
\centering
\includegraphics[scale=0.8]{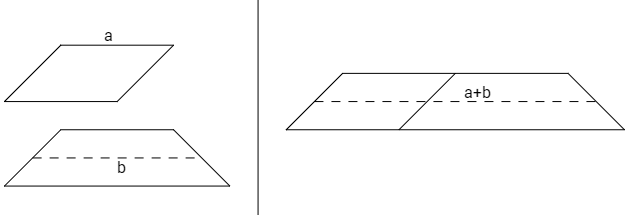}
\caption{К доказательству леммы \ref{draw} -- выполнено условие 4)}\label{k2}
\end{figure}

\begin{figure}[h!]
\centering
\includegraphics[scale=0.8]{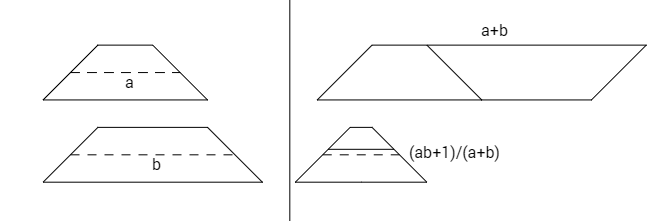}
\caption{К доказательству леммы \ref{draw} -- выполнены условия 3) и 5)}\label{k3}
\end{figure}

\begin{proof}

Лемма проверяется полным перебором. Из гомотетий двух стандартных параллелограммов можно составить стандартный параллелограмм двумя способами, из трапеции и параллелограмма -- трапецию, из двух трапеций -- параллелограмм или трапецию. Все эти способы продемонстрированы на рисунках \ref{k1}-\ref{k3}. При каждом способе стыковки выполняется одно из условий.

Возможно, стоит пояснить, откуда берутся формулы. При выполнении условий 1),3) или 4) мы не производим никаких гомотетий, а просто составляем фигуры вместе -- и сумма средних линий равняется средней линии суммы. При выполнении условия 2) чуть сложнее: над параллелограммами нужно провести гомотетии с коэффициентами $b/(a+b)$ и $a/(a+b)$, чтобы их суммарная высота стала равна единице.

При выполнении условия 5) я записал две формулы, по которым можно вычислить среднюю линию трапеции $C$. Первая демонстрирует применение функции $G$ -- если поставить одну трапецию на другую, их отношение оснований новой трапеции будет произведением их отношений оснований, значит $G(c)=G(a)G(b)$ (где $c$ -- средняя линия трапеции $C$). Формула $c=(ab+1)/(a+b)$ -- это упрощённая формула $G(c)=G(a)G(b)$.
\end{proof}

Так как описывать множества замостимых трапеций мы тоже намерены числами, введём соответствующие обозначения.

\begin{defi}
Пусть $\alpha\subset (1,+\infty)$ -- произвольное множество. Обозначим через $T(\alpha)$ множество таких чисел $b$, что трапеция $t(b)$ замощается трапециями, гомотетичными $t(a_1),..,t(a_k)$, для $a_i\in \alpha$.

Аналогично, через $P(\alpha)$ обозначим множество чисел $b$, таких что параллелограмм $p(b)$ замощается трапециями, гомотетичными $t(a_1),..,t(a_k)$, для $a_i\in \alpha$.
\end{defi}

\begin{lem}[свойства множества $T(\alpha)$ и $P(\alpha)$]\label{base}

\ 

1) $a\in P(\alpha), b\in P(\alpha) \Rightarrow a+b\in P(\alpha)$.

2) $a\in P(\alpha), b\in P(\alpha) \Rightarrow ab/(a+b)\in P(\alpha)$.

3) $a\in T(\alpha), b\in T(\alpha)\Rightarrow a+b\in P(\alpha)$.

4) $a\in P(\alpha), b\in T(\alpha)\Rightarrow a+b\in T(\alpha)$.

5) $a\in T(\alpha), b\in T(\alpha)\Rightarrow (ab+1)/(a+b)\in T(\alpha)$.

6) $a\in P(\alpha), m\in\mathbb Q, m > 0 \Rightarrow am\in P(\alpha)$.

7) Пусть $a\in T(\alpha), \varepsilon > 0$. Тогда существует $c$, которое выражается через $a$ рациональной функцией с рациональными коэффициентами, такое что $c\in T(\alpha), c < 1+\varepsilon$.
\end{lem}

\begin{figure}[h!]
\centering
\includegraphics[scale=0.8]{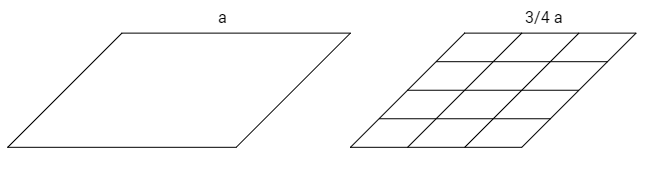}
\caption{К доказательству леммы \ref{base} -- свойство 6}\label{k6}
\end{figure}

\begin{figure}[h!]
\centering
\includegraphics[scale=0.8]{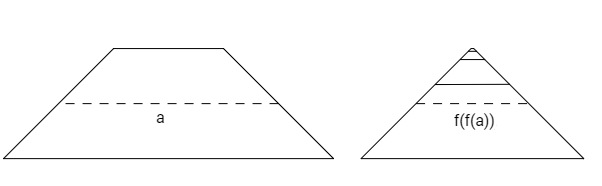}
\caption{К доказательству леммы \ref{base} -- свойство 7}\label{k7}
\end{figure}

\begin{proof}
Первые пять свойств напрямую следуют из пяти конcтрукций, предъявленных в рисунках \ref{k1}-\ref{k3}. Приведу пример: если $a\in P(\alpha), b\in P(\alpha)$, то $p(a),p(b)$ замостимы трапециями, гомотетичными $t(a_1),..,t(a_k)$, для $a_i\in \alpha$. Разделим $p(a+b)$ на параллелограммы $p(a)$ и $p(b)$ как на рис. \ref{k1} и замостим каждый из них. Получим замощение параллелограмма $p(a+b)$, значит $a+b\in P(\alpha)$.

Шестое и седьмое свойство, на самом деле, следуют из предыдущих, однако доказывать их через предыдущие просто неудобно. Поэтому проще предъявить геометрические конструкции.

6) Рациональное число $m$ -- это дробь вида $k/n$. Выложив параллелограммами сетку $k$ на $n$ (см. рис. \ref{k6}) мы получим параллелограмм с отношением основания к высоте равным $am$. После гомотетии это как раз будет $p(am)$.

7) Сложим основаниями две трапеции, гомотетичные $t(a)$. Затем будем складывать получаемые так трапеции много раз, образуя пирамиду из $2^n$ трапеций (см. рис. \ref{k7}). Средняя линия полученной трапеции устремляется к единице при увеличении числа трапеций в конструкции, так как отношение оснований стремится к нулю. Остаётся показать, что средняя линия полученной трапеции будет значением рациональной функции от $a$ с рациональными коэффициентами.

Если трапецию $t(a)$ сложить с гомотетичной ей трапецией как на рисунке \ref{k7}, получится $t((a^2-1)/(2a))$ (это частный случай применения свойства 5). Если обозначить $(a^2-1)/(2a)$ через $f(a)$, то при сложении $2^n$ гомотетий данной трапеций $t(a)$ мы получим
$t(f(f(...f(a)...)))$, где $f$ применяется $n$ раз. Искомое число $c = f(f(...f(a)...))$; оно сколь угодно близко к единице (если в пирамиде было достаточно трапеций), и оно рационально выражено через $a$.
\end{proof}

\section{Доказательство теоремы \ref{my_big} о рациональных средних линиях}\label{me_proof}

\textbf{Теорема \ref{my_big} }(о рациональных средних линиях)

\emph{Пусть $A, B$ -- трапеции единичной высоты с параллельными основаниями, углами $45^o$ при основании и средними линиями $a,b,$ соответственно. Пусть $a\in \mathbb Q$. Тогда трапециями, гомотетичными $A$, можно замостить трапецию $B$, если и только если $b\in\mathbb Q$.}

\begin{proof}

При всех шагах доказательства стоит помнить, что $a > 1, b >1$.

1) В сторону "только если".

Отношение каждого из оснований трапеции $A$ к высоте рационально. Из теоремы \ref{kenyon_big} следует, что любая замощаемая ей трапеция (в частности, $B$), тоже будет иметь рациональные отношения оснований к высоте. Поэтому и средняя линия $b$ будет рациональна.

\bigskip

2) В сторону "если".

Достаточно доказать, что $b\in T(a)$. Рассмотрим следующую цепочку утверждений, использующих свойства множества $T(a)$ из леммы \ref{base}.

Во-первых $a\in T(a)$, т.к. $A$ тривиально замостима сама собой.

По свойству 7) множество $T(a)$ содержит некоторое рациональное число $c$, меньшее $b$.

По свойству 3) выполнено $a+a = 2a\in P(a)$.

Число $(b-c)/2a$ рационально и положительно. По свойству 6) выполнено $2a(b-c)/2a = (b-c)\in P(a)$.

По свойству 4) выполнено $c+(b-c)=b\in T(a)$, что и требовалось доказать.\end{proof}

\section{Вспомогательные леммы для доказательства теоремы \ref{twodim} о средних линиях из $\mathbb Q[\sqrt d]$}\label{hlems}

В следующих параграфах мы работаем в поле $\mathbb Q[\sqrt d]$, где $d > 0, d\in\mathbb Q, \sqrt d\not\in\mathbb Q$.

\begin{defi}
Определим на числах из $\mathbb Q[\sqrt d]$, строго больших единицы, функцию $h(x)$ со значениями в $\mathbb Q\cup\{\infty\}$ следующим образом. Пусть $G(x) = a+b\sqrt d$ для некоторых рациональных $a, b$. Тогда положим $h(x) = a/b$ eсли $b\not = 0$; и $h(x)=\infty$, иначе.

\end{defi}

Сразу замечу, что числа, для которых $h=\infty$ -- это рациональные числа.

Основное свойство функции $h$: условие $G(x)/G(y)\in \mathbb Q$ эквивалентно условию $h(x)=h(y)$. Оно нужно по простой причине -- если у трапеций отношения оснований отличаются в рациональное число раз, то мы можем получить одну из другой, приставив ко второй трапецию с рациональным отношением оснований. Это свойство формализовано в следующей лемме.

\begin{lem}[о числах с равными $h$] \label{h1}
Пусть $x > y$ и $h(x)=h(y)$. Пусть $m\in\mathbb Q$ и $m>1$. Тогда трапециями, гомотетичными трапециям $t(x)$ и $t(m)$, можно замостить трапецию $t(y)$, иными словами, $y\in T(x,m)$.  
\end{lem}
\begin{proof}
Из $h(x)=h(y)$ следует, что $G(y)/G(x)\in\mathbb Q$. По лемме \ref{glem}, функция $G$ возрастает, поэтому $y<x\Rightarrow G(y)<G(x)$. Таким образом, $G(y)/G(x)$ -- рациональное положительное число, меньшее единицы.

На числах из интервала $(0,1)$ определена функция $G^{-1}$. Применив её, мы получим, что $G^{-1}(G(y)/G(x))$,  это рациональное число, большее единицы (по лемме \ref{glem}).

По теореме \ref{my_big}, трапециями, гомотетичными $t(m)$, можно замостить трапецию $t(G^{-1}(G(y)/G(x)))$, потому что у обеих трапеций рациональные средние линии. Таким образом, $G^{-1}(G(y)/G(x))\in T(x,m)$.

Воспользуемся свойством 5) из леммы \ref{base} и тем, что $x\in T(x,m)$ и $G^{-1}(G(y)/G(x))=(xy+1)/(x+y)\in T(x,m)$.

Получим $G^{-1}(G(x)G(G^{-1}(G(y)/G(x)))) = G^{-1}(G(x)G(y)/G(x)) = y\in T(x,m)$.

Что и требовалось доказать.

\end{proof}

У следующих двух лемм есть простой геометрический смысл. Вы можете заглянуть в раздел \ref{geom}, чтобы подробнее о нём узнать (хотя это не обязательно для понимания доказательства).

\begin{lem}[о большом $x$ при фиксированном $h(x)$]\label{h2}
Для любых $q\in \mathbb Q$, $N > 2, \varepsilon > 0$ существует число $x$, такое что $h(x)=q$, $x > N$ и  $|\overline x - \frac{q}{\sqrt d}| < \varepsilon$.

\end{lem}

\begin{proof} Для удобства я принимаю $q+\sqrt d > 0$, в противном случае доказательство проводится аналогично.

Построим искомое число $x$ явно. Для этого мне нужно ввести переменные, через которые его удобно выразить.

Возьмём достаточно малое $\delta > 0$, для которого выполняется следующий ряд условий:
\begin{itemize}
\item $\delta(q+\sqrt d) < 1 - G(N)$
\item $\delta\sqrt d|d-q^2|<d$ 
\item $\frac{\delta (q+\sqrt d)}{d}\left|d-q^2\right| < \varepsilon$
\item $\delta < \frac{1}{q+\sqrt d}$
\item $\frac{1}{q+\sqrt d} - \delta\in \mathbb Q$

\end{itemize}
Нетрудно убедиться, что все числа, которыми мы ограничиваем $\delta$, положительны. Из плотности $\mathbb Q$ следует что мы найдём и $\delta$, удовлетворяющую последнему условию.

Я утверждаю, что условиям леммы удовлетворяет число 

$$x = G^{-1}\left(\left(q+\sqrt d\right)\left(\frac{1}{q+\sqrt d} - \delta\right)\right).$$ Проверим это.
\bigskip

Условие 0: Покажем, что $x$ определён корректно. Конкретно, проверить нужно корректность функции $G^{-1}$, определённой только на интервале $(0,1)$. Из предпоследнего условия на $\delta$ следует, что выражение в скобках больше нуля, а из $\delta > 0$ следует, что оно меньше единицы.

\bigskip
Условие 1: Докажем, что $h(x)=q$. Применив функцию $G$, получаем $G(x) = (q+\sqrt d)(\frac{1}{q+\sqrt d}-\delta)$. По последнему условию на $\delta$, выполнено $(\frac{1}{q+\sqrt d} - \delta)\in \mathbb Q$, поэтому $G(x)$ имеет вид $(q+\sqrt d)a$ для некоторого рационального $a$. Это, по определению, значит что $h(x)=q$.

\bigskip

Условие 2: Докажем, что $x > N$. Действительно, $$G(x) = \left(\frac{1}{q+\sqrt d} - \delta\right)\left(q+\sqrt d\right) = 1 - \delta\left(q+\sqrt d\right)  > 1 - (1-G(N)) = G(N).$$ Так как $G$ - возрастающая, то из $G(x) > G(N)$ следует, что $x > N$.

\bigskip

Условие 3: Докажем, $\left|\overline x - \frac{q}{\sqrt d}\right| < \varepsilon$. Мы знаем что $\frac{1}{q+\sqrt d}-\delta\in \mathbb Q$, поэтому можем сопрячь выражение:
$$\overline x = \overline {G^{-1}\left(\left(\frac{1}{q+\sqrt d} - \delta\right)\left(q+\sqrt d\right)\right)} = $$

$$=\frac{1 + \left(\frac{1}{q+\sqrt d} - \delta\right)\left(q - \sqrt d\right)}{1 - \left(\frac{1}{q+\sqrt d} - \delta\right)\left(q - \sqrt d\right)}= \frac{2q+ \delta(d-q^2)}{2\sqrt d - \delta(d - q^2)}$$

Теперь оценим $\left|\overline x - \frac{q}{\sqrt d}\right|$, после чего используем условия, ограничившие $\delta$:

$$\left|\frac{2q + \delta(d-q^2)}{2\sqrt d - \delta(d-q^2)} - \frac{q}{\sqrt d}\right| = \left|\frac{\delta \left(d-q^2\right)\left(q+\sqrt d\right)}{2d - \delta\left(d-q^2\right)\sqrt d}\right|\le$$

$$\le \left|\frac{\delta (d-q^2)\left(q+\sqrt d\right)}{d}\right| < \varepsilon,$$

что и требовалось доказать.
\end{proof}

Наконец, сформулируем последнюю лемму, необходимую для доказательства теоремы \ref{twodim}.

\begin{lem}\label{h3}
Пусть $a, b > 0$, $a, b \in \mathbb Q[\sqrt d]$, $d\in \mathbb Q, \sqrt d\not\in \mathbb Q$. Пусть также $\overline a > 0$, $\overline b < 0$. Тогда для любого рационального $q$ и любого $M>0$ существует $x$ со следующими свойствами:

1) $h(x) = q$

2) $x>M$

3) $x=ma+nb$ для некоторых положительных рациональных $m,n$.
\end{lem}

\begin{proof}
Воспользовавшись леммой \ref{h2}, найдём $x$, такой что  
$$x > \max\left\{M, \frac{\left(\left|q/\sqrt d\right|+1\right)b}{|\overline b|}, \frac{\left(\left|q/\sqrt d\right|+1\right)a}{|\overline a|}\right\},$$ $$\left|\overline x - q/\sqrt d\right| < 1,\ h(x)=q.$$ Проверим, что тогда следующие $m,n$ удовлетворяют всем указанным в условии 3) свойствам (положительны, рациональны, и $ma+nb = x$).

$$m=\frac{x\overline b - \overline x b}{a\overline b - \overline a b};\ \ \  n=\frac{x\overline a - \overline x a}{b\overline a - \overline b a}$$

1) Дроби определены корректно: используя $a,b,\overline a>0$, $\overline b < 0$ увидим, что знаменатель первой дроби строго меньше нуля, а второй - больше нуля.

2) Имеем $ma+nb = \frac{x\overline ba - \overline x ba - x\overline ab + \overline x ab}{a\overline b - \overline a b} = \frac{x(\overline ba - \overline ab)}{a\overline b - \overline a b}=x$.

3) Имеем $m,n\in \mathbb Q$, так как числители и знаменатели обеих дробей представимы в виде $t-\overline t$ для $t\in\mathbb Q[\sqrt d]$, а значит, имеют вид $q\sqrt d$ для некоторого $q\in\mathbb Q$. Частное двух чисел такого вида рационально.

4) Докажем, что $m > 0$; для $n$ рассуждение симметрично.

Знаменатель, как уже было сказано, меньше нуля. Остаётся доказать $x\overline b < \overline x b$. Однако, левое число меньше нуля ($x>0, \overline b < 0$), поэтому достаточно доказать
$|x\overline b| > |\overline x b|$. Используем оценки на $x$:

$$\left|\overline x b\right| < \left|q/\sqrt d\right||b| + \left|\overline x - q/\sqrt d\right||b| < \left(\left|q/\sqrt d\right|+1\right)|b| < |x||\overline b|.$$

Подходящие $m,n$ найдены, что и требовалось доказать.

\end{proof}

\section{Доказательство теоремы \ref{twodim} }\label{di_proof}
\textbf{Теорема \ref{twodim} } \emph{(о средних линиях из $\mathbb Q[\sqrt d]$)}

\emph{Пусть $d\in\mathbb Q,\sqrt d\not\in\mathbb Q$. Пусть $A,B,C$ -- трапеции единичной высоты c параллельными основаниями, углами $45^o$ при основании, и длины их средних линий равны $a,b,c$. Пусть также $a,b\in\mathbb Q[\sqrt d]$ и $\overline a > 0$, $\overline b < 0$. Тогда трапециями, гомотетичными $A$ и $B$, можно замостить $C$, если и только если $c\in\mathbb Q[\sqrt d]$.}

\begin{proof}

Доказательство в сторону "только если" строится так же, как в теореме \ref{my_big}, за исключением того что вместо $\mathbb Q$ мы используем $\mathbb Q[\sqrt d]$. Ключевым моментом является доказательство в сторону "если" (построение замощения), для которого и были введены новые леммы.

Мы хотим доказать для произвольного $c>1$ из $Q[\sqrt d]$, что $c\in T(a,b)$.

Говоря о <<свойствах>> я, как и раньше, имею в виду свойства из леммы \ref{base}.

Мы имеем $a, b\in T(a,b)$.

По свойству 3) выполнено $2a, 2b \in P(a,b)$.

Теперь применим лемму \ref{h3} к числам $2a,2b$ и найдём $x > (a+b)$, такой что $x = 2am+2bn$ и $h(x) = 0$, для некоторых $m,n\in\mathbb Q$.

Из того, что $x=2am+2bn>(a+b)$ следует, что либо $m$, либо $n$ превышают $\frac12$. Пусть $m>\frac12$ (случай $n>\frac12$ полностью аналогичен).

Тогда по свойству 6) выполнено $2a(m-1/2)\in P(a,b)$ и $2bn\in P(a,b)$.

По свойству 1) имеем $2a(m-1/2)+2bn\in P(a,b)$.

По свойству 4) имеем $a+2a(m-1/2)+2bn = 2am+2bn = x \in T(a,b)$.

Мы знаем, что $h(x)=0$. Значит, $G(x)$ имеет вид $q\sqrt d$ для некоторого $q\in\mathbb Q$. Применяя свойство 5) к $a=b=x$, получим $G^{-1}(G(x)G(x))\in T(a,b)$. Это число рационально, т.к. $G(x)^2=q^2d\in\mathbb Q$, а $G^{-1}$ -- рациональная функция с рациональными коэффициентами. Таким образом, $T(a,b)$ содержит хотя бы одно рациональное число.

Если $c$ рационально, то теорема доказана (сведена к теореме \ref{my_big}). Если же нет -- применим лемму \ref{h3} к числам $2a, 2b, q = h(c), M = a+b+c$ и найдём $x'$, такой что $h(x')=h(c)$ и $x'>c$. Повторяя доказательство того, что $x\in T(a,b)$, получим, что $x'\in T(a,b)$. Теперь у нас есть все условия (важным условием было наличие рационального числа в $T(a,b)$), чтобы воспользоваться леммой \ref{h1}. По лемме \ref{h1}, получим $c\in T(a,b)$, ч.т.д.

\end{proof}

\section{Тривиальные разрезания и связанные с ними леммы}\label{tri}

Определим тривиальные разрезания, упомянутые в формулировке теоремы \ref{log}.

\begin{defi}
Замощение плоской фигуры называется \emph{тривиальным} если выполнено одно из двух условий:

1) Оно состоит только из данной фигуры

2) Существует отрезок, проходящий только по границам фигур замощения и разделяющий фигуру на две, каждая из которых замощена тривиально (определение применяется рекурсивно).
\end{defi}

В связи с новым видом замощений вводим новое обозначение:

\begin{defi}
Пусть $\alpha$ -- множество чисел, больших единицы. Обозначим за $T'(\alpha)$ множество средних линий трапеций, которые можно \emph{тривиально} замостить трапециями, гомотетичными $t(a_1),..,t(a_k)$ для $a_1,..,a_k\in\alpha$. Аналогично введём обозначение $P'(\alpha)$. 
\end{defi}

Очевидно, что $T'(\alpha)\subset T(\alpha)$, $P'(\alpha)\subset P(\alpha)$. Нетрудно также доказать, что пара $(T'(\alpha),P'(\alpha))$ удовлетворяет свойствам из леммы \ref{base}. Достаточно просто повторить доказательство леммы $\ref{base}$, замечая при этом, что при составлении новой стандартной фигуры любым из указанных способов, получается тривиальное замощение.

Для удобства ссылок, я выношу это утверждение в отдельную лемму; доказательство леммы опустим.

\begin{lem}\label{dbase}
Пусть $\alpha\subset\{x\in\mathbb R:x>1\}$. Тогда для множеств $T'(\alpha),P'(\alpha)$ выполнены свойства $1),..,7)$ из леммы \ref{base}.
\end{lem}

Новым результатом, после перехода к тривиальным разрезаниям, становится следующая лемма, в некотором роде, обратная лемме \ref{base}. 
\begin{lem}\label{rbase}
Пусть $\alpha\subset\{x\in\mathbb R:x>1\}$. Пусть $T, P$ -- такая пара множеств, что для них выполняются свойства $1),..,5)$ из леммы \ref{base}, и при этом $\alpha\subset T$.

Тогда $T'(\alpha) \subset T$ и $P'(\alpha) \subset P$.
\end{lem}

\begin{proof}
Пусть $c\in T'(\alpha)$ или $c\in P'(\alpha)$. Это значит, что $t(c)$ или $p(c)$ тривиально замощается трапециями, гомотетичными трапециям $t(a_1),..,t(a_k)$, где $a_i\in \alpha$. Пусть минимальное такое замощение состоит из $n$ фигур. Мы утверждаем, что в этом случае $c$ будет также лежать в $T$ или $P$ соответственно; доказывать это будем индукцией по $n$.

Далее я веду рассуждение только для $c\in T'(\alpha)$, при $c\in P'(\alpha)$ оно проводится аналогично.

База: если $n = 1$, то $c\in \alpha$. Мы знаем, что $\alpha\subset T$, поэтому $c\in T$.

Переход: по определению тривиального разрезания, есть прямой разрез через всю трапецию $t(c)$, разделяющий её на две фигуры, замостимые трапециями $t(a_1),..,t(a_k)$ для $a_1,..,a_k\in\alpha$. Так как разрез идёт по границам фигур, он может проходить только в трёх направлениях -- горизонтально или под углом $45^o$ к горизонтали (в любую сторону).

Перебрав варианты прохождения разреза, мы увидим, что либо он делит трапецию на две как на рисунке \ref{k3}, либо на параллелограмм, как на рисунке \ref{k2}, либо на шестиугольник и треугольник. Однако треугольник было бы невозможно сложить из трапеций, поэтому этот вариант отбрасывается. Разберём два возможных случая:

1) Если трапеция разбита горизонтальным разрезом на две трапеции, гомотетичные стандартным, обозначим эти две трапеции как $t(a)$ и $t(b)$. Тогда $c=(ab+1)/(a+b)$ (это равенство доказано в лемме \ref{draw}).

Заметим, что $a$ и $b$ лежат в $T'(\alpha)$, так как составлены из трапеций, гомотетичных $t(a_1),..,t(a_k)$ для $a_1,..,a_k\in\alpha$. Более того, замощения $t(a)$ и $t(b)$ состоят не более, чем из $n-1$ фигуры, поэтому, по предположению индукции, $a,b\in T$.

Однако $T$ удовлетворяет свойствам из леммы \ref{base}, в том числе, $a,b\in T\Rightarrow (ab+1)/(a+b)\in T$. Таким образом, $c\in T$.

2) Если трапеция разбита на параллелограмм и трапецию, нужно действовать аналогично, но в конце применить свойство 4).

Если $c\in P$, шаг индукции проводится аналогично, но разбираются три варианта разрезания параллелограмма на две стандартные фигуры.

\end{proof}

\section{Теорема \ref{log} о логарифмах}\label{lo_proof}

Доказательство теоремы о логарифмах состоит из громоздкой проверки многих утверждений. Для начала я повторю формулировку и объясню общее направление доказательства.

\textbf{Теорема \ref{log}}

\emph{Пусть $A, B$ -- трапеции единичной высоты c параллельными основаниями, углами $45^o$ при основании и средними линиями $a, b\in \mathbb Q[\sqrt d]$, причём $1 < \overline a < a$. Пусть трапециями, гомотетичными трапеции $A$, можно тривиально замостить трапецию $B$.
Тогда $b$ удовлетворяет следующим условиям:}

(i) $1 < \overline b < b$;

(ii) $\overline b/ b \ge \overline a/a$;

(iii) $\ln G(b)/ \ln \overline {G(b)} \ge \ln G(a)/ \ln \overline{G(a)} $, где $G(x) = (x-1)/(x+1).$
\bigskip

\textbf{Общий план доказательства (см. также раздел \ref{geom} и рисунок \ref{later})}

Введём ещё одно условие:

(i') $0 < \overline b < b$

Обозначим через $T$ множество чисел $b\in \mathbb Q[\sqrt d]$, удовлетворяющих условиям (i),(ii),(iii); через $P$ -- множество чисел $b\in \mathbb Q[\sqrt d]$, удовлетворяющих условиям (i'), (ii).

Мы хотим доказать, что пара $(T,P)$ удовлетворяет свойствам $1),..,5)$ из леммы $\ref{base}$. Тогда по лемме $\ref{rbase}$ мы получим что $T'(a)\in T$ и $P'(a)\in P$ (так как очевидно, что $a\in T$). Как следствие, средняя линия трапеции, замостимой трапецией $A$, будет лежать в $T$, что и требуется доказать.

Однако показать, что пара $(T,P)$ удовлетворяет свойствам $1)-5)$, -- непросто. Для примера рассмотрим свойство $4)$: <<если $x\in P, y\in T$, то $x+y\in T$>>. Чтобы проверить свойство $4)$, нужно убедиться, что число $x+y$ удовлетворяет нетривиальным неравенствам $(i),(ii),(iii)$. Аналогично для каждого свойства. Проверка неравенства $(iii)$ в доказательстве свойства $4)$ оказалось настолько сложной, что её пришлось разнести на три леммы: два вспомогательных неравенства и само неравенство $(iii)$. Поэтому я сначала доказываю эти три леммы, а потом теорему, которая по модулю этих лемм станет сравнительно простой.

\begin{lem}\label{001}
Пусть $z > t > 1$. Тогда $\ln G(t)/\ln G(z) > z/t$.
\end{lem}
\begin{proof}
Значение $G(z)$, как и $G(t)$, лежит в интервале $(0,1)$. Поэтому неравенство $\ln G(t)/\ln G(z) > z/t$ равносильно $t\ln G(t) < z\ln G(z)$.

Рассмотрим функцию $f(x) := x\ln G(x) = x(\ln(x-1)-\ln(x+1))$. Её производные равны:

$$f'(x)=\frac{2x}{(x^2-1)}+\ln\frac{x-1}{x+1},\qquad \ f''(x)=-\frac{4}{(x^2-1)^2}.$$

Функция $f''$ отрицательна, значит, $f'$ строго убывает на луче $(1,+\infty)$. По формуле ясно, что $f'$ стремится к нулю при $x\to+\infty$, а значит $f'$ положительна. Значит, функция $f$ на своей области определения строго возрастает.

Из этого и следует, что для $z > t > 1$ верно $t\ln G(t) < z\ln G(z)$.
\end{proof}
\begin{lem}\label{000}
Для $c\in (0,1)$ определим $F_c(x) = \frac{1+\left(\frac{x-1}{x+1}\right)^c}{1-\left(\frac{x-1}{x+1}\right)^c}$. Тогда $F_c'(x) > 1/c$ для $x> 1$.
\end{lem}
\begin{proof}
Вычислим производную:
$$F_c'(x) = \frac{4c\left(\frac{x-1}{x+1}\right)^c}{(x^2-1)\left(1-\left(\frac{x-1}{x+1}\right)^c\right)^2}.$$

Обозначим $y := \frac{x-1}{x+1}$ и выразим $x$ в формуле $F_c'(x)$ через $y$ (заметим, что условием на $y$ будет $y\in (0,1)$):

$$F_c'(x) = \frac{4cy^c}{\left(\left(\frac{1+y}{1-y}\right)^2-1\right)(1-y^c)^2}=\frac{4cy^c(1-y)^2}{((1+y)^2-(1-y)^2)(1-y^c)^2}=\frac{cy^c(1-y)^2}{y(1-y^c)^2}$$

Мы хотим доказать следующее неравенство:

$$\frac{cy^c(1-y)^2}{y(1-y^c)^2} > \frac 1c$$

Сделаем замену $t=\sqrt y$, тогда на $t$ будет наложено условие $t\in(0,1)$:

$c^2t^{2c}(1-t^2)^2 > t^2(1-t^{2c})^2$

$ct^c(1-t^2)>t(1-t^{2c})$

$(1-t^2c)/(ct^c) < (1-t^2)/t$

$$\frac{\frac1{t^c}-t^{c}}{c} < \frac{\frac1t - t}1.$$

Теперь, если мы докажем, что производная функции $f(x)=(t^{-x}-t^x)/x$ 

на $x\in(0,1]$ больше нуля, то мы докажем утверждение леммы.

Вычислим эту производную:

$$((t^{-x}-t^x)/x)' = \frac{t^{-x}(t^{2x}-x(t^{2x}+1)\ln t - 1)}{x^2}.$$

Мы хотим доказать, что это выражение больше нуля. Поделив на заведомо положительные множители, сведём задачу к следующей:

$$t^{2x}-x(t^{2x}+1)\ln t - 1 > 0.$$

Заменим $u = t^{2x}$, выполнено ограничение $u\in(0,1)$.

$$u-\frac12\ln u(u+1) - 1 > 0;$$

$$\frac12\ln u < \frac{u-1}{u+1}.$$

Можно заметить, что в $u=1$ значения выражений с обоих сторон неравенства совпадают. Чтобы доказать неравенство при $u \in (0,1)$, достаточно проверить $(\frac12\ln u - \frac{u-1}{u+1})' > 0$.

$$(\frac12\ln u - \frac{u-1}{1+u})' = \frac{(u-1)^2}{2u(u+1)^2} > 0.$$

Из этого неравенства следует утверждение леммы, ч.т.д.

\end{proof}

\begin{lem}\label{043}
Напомню, что $P$ -- множество чисел $b\in \mathbb Q[\sqrt d]$, удовлетворяющих условиям $(i',ii)$, а $T$ -- множество чисел $b\in \mathbb Q[\sqrt d]$, удовлетворяющих условиям $(i,ii,iii)$.

Тогда если $x\in T, y\in P$, то $x+y$ удовлетворяет условию (iii), т.е. $$\ln G({x+y})/\ln G(\overline{x+y}) > \ln G(a)/ \ln {G(\overline{a})}.$$
\end{lem}

\begin{proof}
Обозначим $c:=\ln G(a)/ \ln {G(\overline{a})}$. Заметим, что по лемме \ref{glem} из $a>\overline a > 1$ следует, что $c\in (0,1)$. Условие $\ln G(b)/\ln G(b') = c$ в области $1 < b' < b$ можно переписать в виде:

$$\ln \Bigl(\frac{b-1}{b+1}\Bigr) = c\ln\Bigl(\frac{b'-1}{b'+1}\Bigr) \Leftrightarrow$$

$$\frac{b-1}{b+1} = \Bigl(\frac{b'-1}{b'+1}\Bigr)^c\Leftrightarrow$$

$$b = \frac{1+(\frac{b'-1}{b'+1})^c}{1-(\frac{b'-1}{b'+1})^c} = F_c(b'),$$

где функция $F_c$ определена в лемме \ref{000}.

Пары $b, b'$, удовлетворяющие $\ln G(b)/\ln G(b') = c$, -- это график функции $b = F_c(b')$.

\smallskip

Так как $x\in T$, то $\ln G(x)/ \ln {G(\overline{x})} \ge c$.

Если в выражении $\ln G(x_0)/ \ln {G(\overline{x})}$ устремить $x_0$ к $+\infty$, то само выражение будет непрерывно уменьшаться, стремясь к нулю. Значит, найдётся такой $x_0 > x$, что $\ln G(x_0)/ \ln {G(\overline{x})} = c$. Что равносильно, $F_c(\overline x) = x_0 > x$.

Введём функцию $K(z) = (z-\overline x)\frac{y}{\overline{y}}+ x$. Заметим, что $K(\overline x)= x$ и $K(\overline{x+y})={x+y}$.

Покажем, что на отрезке от $\overline x$ до $\overline{x+y}$ производная функции $F_c$ больше, чем производная функции $K$. Действительно, $F_c' > 1/c = \ln {G(\overline{a})}/\ln G(a) > a/\overline{a} \ge {y}/\overline{y} = K'$. Здесь использованы леммы \ref{000}, \ref{001}, а также тот факт, что $y\in P$.

Мы уже получили, что $F_c(\overline{x}) > x = K(\overline{x})$, и благодаря неравенству на производные, имеем $F_c(\overline{x+y}) > K(\overline{x+y}) = x+y$.

Рассмотрим тождество $\ln G(F_c(z))/\ln G(z)=c$ (оно равносильно $F_c(z)=F_c(z)$, как показано в начале доказательства). Подставим в него $\overline{x+y}$:
$$\ln G(F_c(\overline{x+y}))/\ln G(\overline{x+y}) = c.$$
Функция $\ln G$ -- убывающая. Как показано в предыдущем абзаце, $x+y<F_c(\overline{x+y})$. Используем это:
$$\ln G(x+y)/\ln G(\overline{x+y}) > c.$$
Вспоминая, что $c=\ln G(a)/\ln G(\overline a)$, получим $\ln G({x+y})/\ln G(\overline{x+y}) > \ln G(a)/ \ln (G(\overline{a})$, ч.т.д.
\end{proof}

Наконец, мы имеем возможность перейти к доказательству теоремы. Я определил множества $T,P$ выше, поэтому сразу перейду к доказательству того, что они удовлетворяют пяти нужным свойствам из леммы \ref{base}.

\begin{proof}[Доказательство теоремы \ref{log}]

\

Свойство 1. Пусть $x, y\in P$. Доказываем, что $x+y\in P$, то есть $(x+y)$ удовлетворяет условиям $(i',ii)$, наложенным на множество $P$.

1(i'). Мы имеем $0<\overline x<x, 0 < \overline y < y$. Сложив, получаем $0 < \overline {(x+y)} < (x+y)$.

1(ii). Мы имеем $x,y > 0$, поэтому дроби $x/(x+y)$ и $y/(x+y)$ положительны и дают единицу в сумме. Домножим на эти коэффициенты неравенства из условия $(ii)$, получим

$$\frac{\overline a}{a} \le \frac{\overline x}{x}\frac{x}{(x+y)} + \frac{\overline y}{y}\frac{y}{(x+y)} = \frac{\overline{x+y}}{x+y},$$

ч.т.д.

\bigskip

Свойство 2. Пусть $x, y\in P$. Мы доказываем, что $(xy)/(x+y)\in P$.

Воспользуемся тождеством $\frac{y\overline y} {(x+y)(\overline{x+y})}x + \frac{x\overline x}{(x+y)(\overline{x+y})}y = \frac{xy}{x+y}$.

У нас есть условия $x>\overline x > 0$, $y>\overline y > 0$, из чего следует, что коэффициенты перед $x$ и $y$ в этом тождестве положительны. Коэффициенты будут рациональны, так как являются произведением сопряжённых друг другу чисел из $\mathbb Q[\sqrt d]$.

Заметим, что если $x\in P$, то для любого положительного рационального числа $a$ верно $ax\in P$ (условия (i'),(ii) проверяются непосредственно).

Тогда мы можем домножить $x,y$ на рациональные положительные числа и получить $\frac{y\overline y} {(x+y)(\overline{x+y})}x\in P$, $\frac{x\overline x}{(x+y)(\overline{x+y})}y\in P$. Используя доказанное выше свойство 1, получаем $\frac{xy}{x+y}\in P$.

\bigskip

Свойство 3. Пусть $x, y\in T$. Доказываем, что $x+y\in P$.

Можно заметить, что неравенство $(i)$ сильнее неравенства $(i')$, поэтому $T\subset P$. Из этого следует, что $x, y\in P$ и, по свойству 1, получаем $x+y\in P$.

\bigskip

Свойство 4. Пусть $x\in T, y\in P$. Доказываем, что $x+y\in T$.

4(i). Аналогично 1(i').

4(ii). Аналогично 1(ii).

4(iii). Доказано в лемме \ref{043}.

\bigskip

Свойство 5. Пусть $x\in T, y\in T$. Доказываем, что $G^{-1}(G(x)G(y))=\frac{xy+1}{x+y}\in T$.

Все пункты будут построены по одному принципу: берём неравенство, следующее из свойств $(i,ii,iii)$ и применяем несколько преобразований.

\bigskip

5(i) Так как $\overline x,\overline y > 1$, то $(\overline x-1)(\overline y - 1) > 0\Rightarrow$

$\overline {xy} + 1 > \overline{x+y}\Rightarrow$

$\frac{\overline{xy}+1}{\overline{x+y}} > 1$. (доказана первая часть неравенства (i)).

Далее, $(y\overline y -1)(x-\overline x) + (x\overline x -1)(y-\overline y) > 0\Rightarrow$

$xy\overline y + xy\overline x +\overline x + \overline y > \overline{xy}y +\overline{xy}x + x + y\Rightarrow$

$(\overline{x+y})(xy+1) > (x+y)(\overline{xy}+1)\Rightarrow$

$\frac{xy+1}{x+y} > \frac{\overline{xy}+1}{\overline{x+y}}$ (доказана вторая часть неравенства (i)).

\bigskip

5(ii) По условию (ii), задающему множество $T$, выполнены неравенства $\overline x/x \ge \overline a/a$ и $\overline y/y \ge \overline a/a$. Домножая обе части неравенств на равные числа, мы получаем $\overline x a - x\overline a \ge 0$, $\overline y a - y\overline a \ge 0\Rightarrow$

$x\overline x(\overline ya-y\overline a) + y\overline y(\overline xa-x\overline a) + (xa-\overline x\overline a) + (ya-\overline y\overline a) > 0\Rightarrow$

$\overline {xy}xa +\overline{xy}ya+xa+ya > xy\overline{xa}+xy\overline{ya} + \overline{xa}+\overline {ya}\Rightarrow$

$(\overline{xy}+1)(x+y)a > (xy+1)(\overline{x+y})\overline a\Rightarrow$

$\frac{(\overline {xy}+1)/(\overline {x+y})}{(xy+1)/(x+y)} = \frac{\overline {G^{-1}(G(x)G(y))}}{G^{-1}(G(x)G(y))} \ge \frac{\overline a}{a}.$

\bigskip

5(iii) Если в условие (iii) вместо $b$ подставить число $G^{-1}(G(x)G(y)$, то неравенство раскроется следующим образом: $$\ln G(G^{-1}(G(x)G(y)))/\ln G(G^{-1}(\overline {G(x)G(y)})) = \frac{\ln G(x)+\ln G(y)}{\ln(\overline {G(x)}) + \ln(\overline {G(y)})}$$. 

Зная, что $\frac{\ln G(x)}{\ln\overline {G(x)}}\ge \frac{\ln G(a)}{\ln\overline {G(a)}}$ и $\frac{\ln G(y)}{\ln\overline {G(y)}}\ge \frac{\ln G(a)}{\ln\overline {G(a)}}$, мы получаем

$\frac{\ln G(x)+\ln G(y)}{\ln(\overline {G(x)}) + \ln(\overline {G(y)})}\ge \frac{\ln G(a)}{\ln\overline {G(a)}}$ по аналогии с пунктом 1(ii).

\bigskip

По лемме $\ref{rbase}$ мы получаем что $T'(a)\in T$ и $P'(a)\in P$ (т.к. очевидно что $a\in T$). Как следствие, средняя линия трапеции, тривиально замостимой гомотетиями трапеции $A$, лежит в $T$, что и требовалось доказать.

\end{proof}

В этом разделе также докажем предложение \ref{last}.

\begin{proof}[Доказательство предложения \ref{last}]
Мы хотим найти бесконечное количество чисел $b\in \mathbb Q[\sqrt d]$, таких что $b\in T'(a)$ и $\ln G (b)/\ln G(\overline b) = \ln G (a)/\ln G(\overline a)$.

Определим $b_1 = a, b_{n+1} = G^{-1}(G(a)G(b_n))$. Попарная различность всех $b_i$ следует из геометрического смысла условия 5) в лемме \ref{draw}: трапеция $t(b_{n+1})$ получается при сложении верхним и нижним основанием двух трапеций, гомотетичных $t(a)$ и $t(b_n)$, а значит, её отношение оснований меньше. Значит, и её средняя линия $b_{n+1}$ меньше, чем $b_n$.

Докажем по индукции, что $b_i\in T'(a)$. Для $b_1 = a$ это очевидно. Переход делается при помощи свойства 5) из леммы \ref{dbase}: из $a\in T'(a), b_n\in T'(a)$ следует $b_{n+1}\in T'(a)$.

Докажем по индукции, что $\ln G (b_{i})/\ln G(\overline b_{i}) = \ln G (a)/\ln G(\overline a)$. Для $b_1=a$ это, опять же, очевидно. Сделаем шаг индукции от $b_n$ к $b_{n+1}$. По предположению индукции мы также знаем, что $\frac{\ln G(b_n)}{\ln G(\overline b_n)} = \frac{\ln G(a)}{\ln G(\overline a)}$, или, равносильно, $\frac{\ln G(b_n)}{\ln G(a)} = \frac{\ln G(\overline b_n)}{\ln G(\overline a)}$. Воспользуемся этим.

$$\frac{\ln G (b_{n+1})}{\ln G(\overline b_{n+1})} = \frac{\ln G(a)+\ln G(b_n)}{\ln G(\overline a)+\ln G(\overline b_n)}= \frac{\ln G(a)\left(1+\frac{\ln G(b_n)}{\ln G(a)}\right)}{\ln G(\overline a)\left(1+\frac{\ln G(\overline b_n)}{\ln G(\overline a)}\right)}  = \ln G (a)/\ln G(\overline a).$$

Мы доказали, что каждое $b_i$ лежит в $T'(a)$ и удовлетворяет равенству $\ln G (b_{i})/\ln G(\overline b_{i}) = \ln G (a)/\ln G(\overline a)$, ч.т.д.

\end{proof}

\section{Геометрическая интерпретация}\label{geom}

Несмотря на то, что доказываемые теоремы изначально утверждают геометрические факты (и первые две теоремы можно использовать, чтобы явно конструировать замощения трапеций), здесь речь пойдёт не об этой геометрической интерпретации. Мне хотелось бы неформально пояснить, как были придуманы доказательства (и отчасти формулировки) теорем \ref{twodim} и \ref{log}.

Для начала вернусь к лемме \ref{base} и выделю из имеющихся свойств 1) и 6). Числа, лежащие в $P$, можно складывать и домножать на рациональные. Это наводит на мысль, что параллелограммам можно сопоставить некоторое векторное пространство. И если при $P\subset \mathbb Q$ это было ещё не так важно, то при $P\subset \mathbb Q[\sqrt d]$ такой подход сильно помогает понять процесс.

Как было сказано ранее, стандартный параллелограмм задаётся длиной основания. Пойдём дальше: будем изображать параллелограмм $p(a)$ точкой $(\overline a, a)$, а трапецию $t(b)$ точкой $(\overline b, b)$. Оси координатной плоскости я буду называть $\overline y, y$, чтобы лучше видеть, что и по какой оси мы отмечаем.

Так, на чертеже \ref{starting} (для $d=2$) я отметил точку $T$, соответствующую $t(4)$ и $P$, соответствующую $p(1+\sqrt 2)$. Нетрудно заметить, что точки, соответствующие всевозможным трапециям вида $t(b)$ где $b\in\mathbb Q[\sqrt d]$, будут всюду плотны на полуплоскости $y > 1$, а соотвествующие параллелограммам -- на $y > 0$. Для удобства я обозначаю пунктиром прямую, ниже которой не могут находиться точки, соответствующие трапециям.

\begin{figure}[h!]
\centering
\includegraphics[scale=0.6]{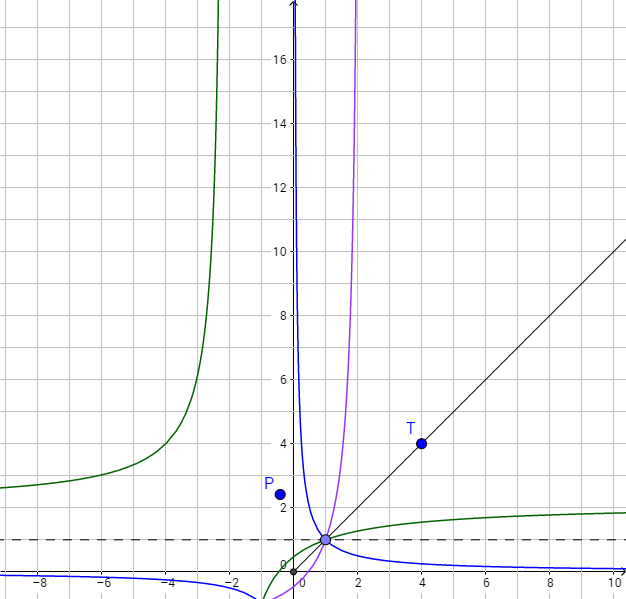}
\caption{Плоскость $(\overline y,y)$ с линиями $h=-3,h=0,h=3,h=\infty$}\label{starting}
\end{figure}

В этом разделе понятие $T(\alpha)$ резонно переопределить не как множество чисел, а как множества точек вида $(\overline y, y)$, таких что трапеция $t(y)$ замощается трапециями $t(a_1),..,t(a_k)$ для $a_1,..,a_k\in\alpha$. Для так определённых множеств $T(\alpha),P(\alpha)$ будет несложно переписать свойства из леммы \ref{base}: например, свойство 1) будет выглядеть как $(\overline a,a)\in P(\alpha), (\overline b,b)\in P(\alpha)\Rightarrow (\overline a+\overline b, a+b)\in P(\alpha)$.

Линии, которые я отметил на чертеже, -- это линии уровня функции $h(\overline y, y) := h(y)$. В частности, луч, на котором лежит $T$, -- это линия уровня $h=\infty$, рациональные числа. Остальные три линии соответствуют $h=-3,h=0$ и $h=3$ (зелёная, синяя, фиолетовая).

Здесь я опускаю формальное доказательство, но линия уровня $h(x)=k$ -- это часть гиперболы $(y+\frac{k}{\sqrt 2})(\overline y - \frac{k}{\sqrt 2})=1-\frac{k^2}{2}$. В частности, у таких гипербол есть вертикальные ассимптоты -- сколь угодно большой $y$ при стремящемся к константе $\overline y$. Неочевидно, что множество точек вида $(y,\overline y)$ на гиперболе окажется плотным -- однако, лемма \ref{h2} это в точности утверждение о том, что на ней есть сколь угодно высокие точки такого вида.

Лемма же \ref{h1} сводится к следующему -- если в $T(\alpha)$ лежит некоторая точка $M$ с луча $\{y=\overline y, y > 1\}$ и некоторая точка $X$ с гиперболы $h(y)=q$, то в ней также лежат все точки (вида $(y,\overline y)$ для $y\in\mathbb Q[\sqrt d]$), находящиеся на гиперболе $h(y)=q$ ниже, чем $X$.

Теперь мы в состоянии осознать геометрический смысл всего доказательства теоремы \ref{twodim}. В ней сказано, что трапеции $t(a), t(b)$ с $\overline a > 0$ и $\overline b < 0$ способны замостить любую другую трапецию. Соответствующие им точки $A,B\in T(a,b)$ расположены по разные стороны от оси $\overline y = 0$ (здесь и далее см.рис \ref{great}).

\begin{figure}[h!]
\centering
\includegraphics[scale=0.6]{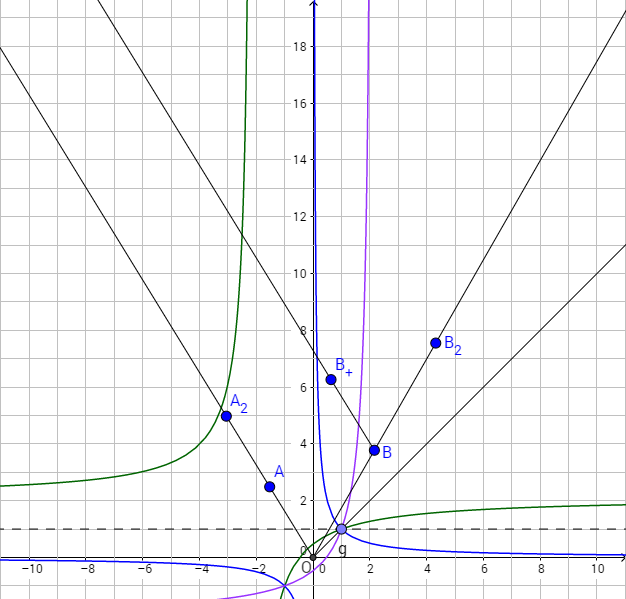}
\caption{Геометрическое толкование доказательства теоремы \ref{twodim}}\label{great}
\end{figure}

Сначала мы пользуемся свойством 3) из леммы \ref{base} -- векторно складываем каждую из этих точек саму с собой, получаем точки $A_2, B_2\in P(\alpha)$, расположенные по разные стороны от оси $\overline y = 0$. Очевидный факт: угол, натянутый на вектора $A_2,B_2$ пересекает каждую из гипербол, содержащих множества $\{h(y)=q\}$ и, как следствие, для любого $q$ этот угол содержит точку $(\overline x, x)$, такую что $h(x) = q$.

У множества $P(\alpha)$ есть крайне удобные свойства 1) и 6) из которых следует его замкнутость относительно векторного сложения и домножения вектора на положительное рациональное число. Из этого следует, что $P(\alpha)$ будет содержать \emph{все} точки вида $(\overline y, y)$ для $y\in\mathbb Q[\sqrt d]$, лежащие внутри угла $AOB$ (шаг не совсем очевиден, и доказывается аналогично лемме \ref{h3}). А если воспользоваться свойством (3), по которому сумма точки из $T(\alpha)$ и $P(\alpha)$ лежит в $T(\alpha)$, то мы можем ещё и доказать, что точки угла $B_+BB_2$ лежат в $T(\alpha)$ (здесь отрезок $B_+B$ построен параллельно $OA$).

Именно эта логика -- представление необходимой точки угла в виде суммы точки из $T(\alpha)$ и линейной комбинации точек из $P(\alpha)$ -- стоит за леммой $\ref{h3}$. В самой лемме нет аналога <<угла>>, потому что мы сразу нацелены на более важный результат -- для каждой гиперболы вида $h(y)=q$ все её достаточно высокие точки лежат в $T(\alpha)$. И, если мы найдём в $T(\alpha)$ точку вида $(m,m)$ ($m\in\mathbb Q$), то, исходя из леммы \ref{h1} или описанного выше свойства гипербол, в $T(\alpha)$ лежат все трапеции вида $t(b)$ для $b\in\mathbb Q[\sqrt d]$.

Чтобы доказательство было закончено, надо доказать, почему $T(\alpha)$ содержит хотя бы одну точку вида $(m,m)$ для рационального $m$. В доказательстве теоремы \ref{twodim} разобрано, почему если $x\in T(\alpha)$ и $h(x)=0$, то $T(\alpha)$ содержит некоторое рациональное число. На чертеже этот факт не нагляден, поэтому не был показан.
\bigskip

Теперь, когда мы освоились с геометрическим представлением множества трапеций, мне бы хотелось перейти к геометрической стороне теоремы \ref{log} и, в частности, пояснить, чем интересно условие, заданное трансцендентной функцией.

\begin{figure}[h!]
\centering
\includegraphics[scale=0.7]{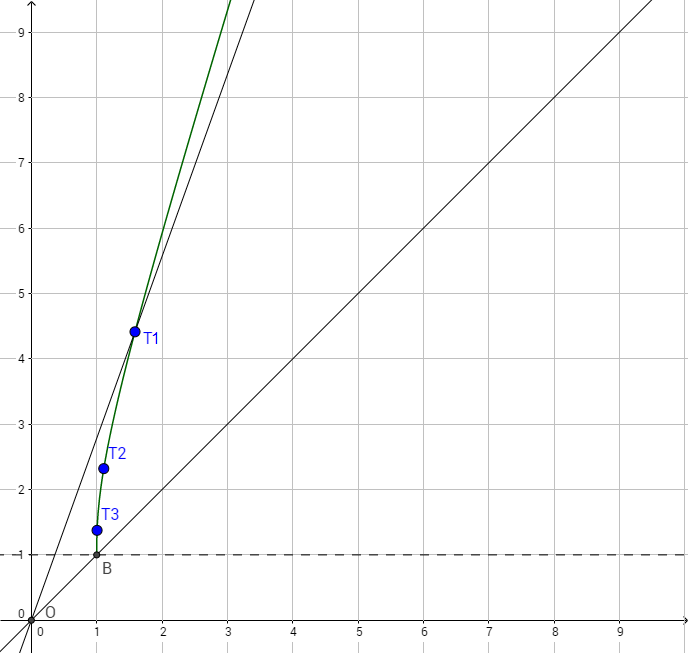}
\caption{К теореме \ref{log}}\label{later}
\end{figure}

В этом разделе под $T'(a)$ я так же понимаю множество \emph{точек}, соответствующих тривиально замостимым трапециям.

Мы рассматриваем множество $T'(a)$ для $a>\overline a > 1$. На рисунке \ref{later} координатам $(\overline a, a)$ для $a=3+\sqrt2$ соответствует точка $T_1$.

Луч $OB$ задаётся уравнением $\overline y = y$, луч $OT_1$ задаётся уравнением $\overline y/y = \overline a/a$. Кривая $T_1T_2B$ задана уравнением $\ln G (y)/\ln G(\overline y) = \ln G (a)/\ln G(\overline a)$ или же равносильным уравнением $y = F_c(\overline y)$. Здесь обозначение $F_c$ взято из леммы \ref{000}, равносильность двух уравнений доказана в самом начале доказательства леммы \ref{043}.

В доказательстве теоремы \ref{log} рассматриваются множества чисел $T$ и $P$. Из неравенств на эти множества очевидна геометрическая интерпретация -- множеству $P$ соответствуют точки внутри угла $T_1OB$, множеству $T$ -- точки, лежащие внутри того же угла, но ниже кривой $T_1T_2B$. Доказывается, что эти два множества (в данной интерпретации, это множества точек) удовлетворяют всем свойствам из леммы \ref{base}. В частности, пункт $4(iii)$ доказательства переформулируется как: <<Для точки $a$ из области $T$ и точки $b$ из области $P$ доказать, что $a+b$ будет лежать ниже кривой $T_1T_2B$>>. Геометрически очевидно, что это стоит доказывать через тот факт, что наклон кривой $T_1T_2B$ в каждой точке превышает наклон прямой $OT_1$. Это и делается в лемме \ref{000}. Очевидная мысль использовать выпуклость функции в этой ситуации не слишком помогает -- выпуклость так заданной кривой тоже доказывается через сложные вычисления.

Перейдём к геометрическим соображениям, выходящим за пределы доказанных нами теорем. Я хотел бы обратить внимание на участок кривой $BT_1$. Сама кривая описывается трансцендентным уравнением -- это график функции $F_c(x)$ из леммы \ref{000}. Однако, по предложению \ref{last}, на этой кривой лежит бесконечное количество точек, принадлежащих <<сетке>> $\{(\overline y, y):y\in\mathbb Q[\sqrt d]\}$ (на рис. \ref{later} отмечены $T_2,T_3$). Я не знаю способа доказать, что на этом участке кривой больше нет точек из этого множества. Также я не вижу очевидного способа найти на кривой точки сетки выше, чем $T_1$. Это затруднение интуитивно понятно: точки вида $T_n$ соответствуют сложениям трапеции из нескольких гомотетий трапеции $t(a)$ (см. доказательство леммы \ref{last}), но если такую операцию обратить (разрезать трапецию на $n$ подобных друг другу), то отношение оснований полученных трапеций будет лежать за пределами $\mathbb Q[\sqrt d]$.

Таким образом, необходимое условие из теоремы \ref{log} может быть достаточным, но даже если оно не достаточное, указанная трансцендентная кривая проходит через бесконечное число точек на границе множества <<реализуемых>> точек $(b,\overline b)$. Поэтому трансцендентность ограничивающей кривой может считаться достоверно установленной новой особенностью, пока не появлявшейся в работах по разрезанию многоугольников на подобные.


\begin{thebibliography}{10}

\bibitem{BST} R.~L.~Brooks, C.~A.~B.~Smith, A.~H.~Stone, and W.~T.~Tutte, \textrm{The dissection of rectangles into squares}, Duke Math.~J.~\textbf{7} (1940), 312--340.

\bibitem{german} M.~Dehn, \textrm{\"Uber die Zerlegung von Rechtecken in Rechtecke}, Math.~Ann.~\textbf{57} (1903), 314--332.

\bibitem{FLR2} C.~Freiling, M.~Laczkovich, and D.~Rinne,
    \textrm{Rectangling a rectangle},
    Discr.~Comp.~Geom.~\textbf{17} (1997), 217-225.
    
\bibitem{FL1} C. Freiling and D. Rinne, Tiling a square with similar rectangles, Math. Res. Lett. 1 (1994),
547–558;

\bibitem{Geo} A. Georgakopoulos, The Boundary of a Square Tiling of a Graph coincides with the Poisson Boundary, Invent.Math., 203(3), 773-821, 2016.

\bibitem{ken} R.~Kenyon, \textrm{Tilings and discrete Dirichlet problems}, Israel J.~Math.~\textbf{105:1} (1998), 61--84.

\bibitem{FL2}M. Laczkovich and G. Szekeres, Tiling of the square with similar rectangles,
Discr. Comp. Geom. 13 (1995), 569–572

\bibitem{Sharov} F. Sharov, \textrm{Dissection of a rectangle into rectangles with given side ratios}, Mat. Prosveschenie 3rd ser. \textbf{20} (2016), 200--214 [arXiv:1604.00316].
    
\bibitem{PrSk} M.~Skopenkov, M.~Prasolov and S.~Dorichenko, \textrm{Dissections of a metal rectangle}, Kvant \textbf{3} (2011), 10--16 [arXiv:1011.3180].

\bibitem{triangles} B.~Szegedy, \textrm{Tilings of the square with similar right triangles}, Combinatorica \textbf{21:1} (2001), 139--144.

\end{thebibliography}
\end{document}